\documentclass[reqno]{amsart}

\usepackage[T1]{fontenc}
\usepackage{amsmath,amssymb,amsxtra,amscd}
\usepackage{amsthm}
\usepackage[all]{xy}
\usepackage{enumerate}
\usepackage[margin=2.5cm]{geometry}

\newtheorem{thm}{Theorem}[section]
\newtheorem{prop}[thm]{Proposition}
\newtheorem{lem}[thm]{Lemma}
\newtheorem{cor}[thm]{Corollary}

\newtheorem{conj}[thm]{Conjecture}

\theoremstyle{definition}
\newtheorem{defn}[thm]{Definition}
\newtheorem{rmk}[thm]{Remark}

\DeclareMathAlphabet\EuR{U}{eur}{m}{n}
\SetMathAlphabet\EuR{bold}{U}{eur}{b}{n}

\newcommand{\Z}{\mathbb{Z}}
\newcommand{\pLoc}{_{(p)}}

\newcommand{\F}{\mathcal{F}}

\renewcommand{\H}{\mathcal{H}}
\newcommand{\K}{\mathcal{K}}
\newcommand{\Hom}{\operatorname{Hom}}

\newcommand{\Aut}{\operatorname{Aut}}
\newcommand{\mor}{\operatorname{mor}}
\newcommand{\defeq}{\overset{\text{\textup{def}}}{=}}
\newcommand{\Burnside}{\mathbf{A}}
\newcommand{\module}{\operatorname{mod}}

\newcommand{\incl}{\rm incl}
\newcommand{\id}{{\rm id}}

\newcommand{\Orb}{\mathsf{Orb}}
\newcommand{\PreOrb}{\mathsf{Pre\text{-}Orb}}

\newcommand{\la}{\langle}
\newcommand{\ra}{\rangle}

\begin{document}

\author{Sejong Park}
\address{Institute of Mathematics, University of Aberdeen, Aberdeen AB24 3UE, UK}
\email{s.park@abdn.ac.uk}
\author{K\'ari Ragnarsson}
\address{Department of Mathematical Sciences, Depaul University, 2320 N.~Kenmore Avenue, Chicago, IL 60614, USA}
\email{kragnars@math.depaul.edu}
\author{Radu Stancu}
\address{LAMFA, Universit\'e de Picardie, 33 rue St Leu, 80039, Amiens, France}
\email{radu.stancu@u-picardie.fr}

\thanks{The second author is partially supported by NSF grant DMS-1007619.}

\title{On the composition product of saturated fusion systems}

\date{February 25th, 2011}

\begin{abstract}
We say that a fusion system is the composition product of two subsystems if every morphism can be factored as a morphism in one fusion system followed by a morphism in the other. 
We establish a relationship between the characteristic idempotent of a saturated fusion system that is the composition product of saturated subsystems and the characteristic idempotents of the component systems. Consequently we obtain a compatibility result for transfer through the composition product and transfer through the component systems.
\end{abstract}

\maketitle

\section{Introduction and statement of results}
Given a finite group $G$ with subgroups $H$ and $K$, we write $G = HK$ if every element $g \in G$ can be written (non-uniquely) as $g = hk$ with $h \in H$ and $k \in K$. This is a rather special situation, as the composition product $HK$ for arbitrary subgroup $H, K \leq G$ is generally not even a group. Setting $L \defeq H \cap K$, the condition $G = HK$ is equivalent to having an isomorphism of $(K,H)$-bisets 
\begin{equation} \label{eq:G=HK}
  H \times_L K \cong {}_H G_K  \, ,
\end{equation}
where the subscripts denote restriction. Bisets act on Mackey functors and it follows that for a globally-defined contravariant Mackey functor $M$,
the maps
\begin{equation}  \label{map:HGK}
 M(H) \xrightarrow{tr_H^G} M(G) \xrightarrow{res_K^G} M(K) 
\end{equation}
and
\begin{equation} \label{map:HLK} 
 M(H) \xrightarrow{res_L^H} M(L) \xrightarrow{tr_L^K} M(K), 
\end{equation}
where $tr$ denotes transfer and $res$ denotes restriction, agree. This can be read as a very special case of the double coset formula in which there is only one double coset $G = H 1 K$.

In this note we carry the discussion in the preceding paragraph over to fusion systems. The motivation for this is twofold: First, the results proved here provide the compatibility of transfer needed in \cite{DGPS:YoshidaAndTate}, and although a proof for the special case needed there appears in  \cite{DGPS:YoshidaAndTate}, we now set the results in a broader context and give a more conceptual proof.  Second, this work is part of an ongoing project to reformulate fusion-theoretic phenomena in terms characteristic idempotents, which was initiated in \cite{KR:ClSpectra} and \cite{KRStancu:SFSasIdem}. 

We now proceed to state our main result, Theorem \ref{thm:main}, and record some interesting consequences. For the sake of brevity we assume the reader is familiar with the basic theory of fusion systems and define only the new concept of composition product used in this paper.  We follow the conventions established by Broto--Levi--Oliver in \cite{BLO2} for fusion systems, but also refer the reader to \cite{Puig:FrobeniusCategories} for the original reference. Other notation and terminology used in the statement of results will be recalled in Section \ref{sec:Notation}. 

In analogy with the situation $G= HK$ above, we make the following definition.
\begin{defn} \label{def:CompProduct}
Let $\F$ be a fusion system on a finite $p$-group $S$, and let $\H$ and $\K$ be fusion subsystems of $\F$ on subgroups $R$ and $T$ of $S$, respectively. We say that $\F$ is the \emph{composition product} of $\H$ and $\K$, and write ${}_R \F_T = \H \K$, if $S = RT$ and every morphism $P \xrightarrow{\varphi} Q$ in $\F$ with $P \leq T$ and $Q \leq R$ can be written as $\varphi = \eta \circ \kappa$ with $\kappa$ in $\K$ and $\eta$ in $\H$. 
\end{defn}

Just as in the group case, this is a special situation, and generally one should not expect to be able to take a composition product of two arbitrary fusion systems, even if they are defined on the same $p$-group. Note also that when $\F, \H$ and $\K$ are all defined on the same $p$-group $S$, the condition $\F = \H\K$ is stronger than saying that $\F$ is generated by the fusion systems $\H$ and $\K$ as the order of morphisms is specified in a composition product. However, if $\H$ or $\K$ is weakly normal in $\F$, then the conditions are equivalent. 

In the fusion setting, the biset $G$ in the discussion above is replaced by the characteristic idempotent $\omega_\F$ of a fusion system $\F$. The analogue of \eqref{eq:G=HK} is 
 \[ {}_R (\omega_\F)_T = \omega_\H \circ {}_R S_T  \circ \omega_\K \, ,\]
with subscripts denoting restriction. We conjecture that this is again equivalent to ${}_R \F_T=\H\K$, stated formally as follows.

\begin{conj} \label{running-conjecture}
Let $\F$ be a saturated fusion system on a finite $p$-group $S$. Let $R$ and $T$ be subgroups of $S$ such that $S = RT$, and let $\H$ and $\K$ be saturated fusion subsystems of $\F$ on $R$ and $T$, respectively.  Then ${}_R \F_T = \H \K$ if and only if
 \[ [R,\id]_S^R \circ \omega_\F \circ [T,\incl]_T^S = \omega_\H \circ [R,\id]_S^R \circ [T,\incl]_T^S \circ \omega_\K \, .\]
\end{conj}

The most important instance of this conjecture is when $R = S$ (see Conjecture \ref{conj-right-side}), and our main result is a proof of this case under a normalcy condition on $\K$. 

\begin{thm} \label{thm:main}
Let $\F$ be a saturated fusion system on a finite $p$-group $S$.  Let $\H$ be a saturated subsystem of $\F$ on $S$, and let $\K$ be a saturated subsystem of $\F$ on a subgroup $T$. Assume $\K$ is weakly normal in $\F$. Then ${}_{S}\F_T = \H\K$ if and only if
\[
	\omega_\F \circ [T,\incl]_T^S = \omega_\H \circ [T,\incl]_T^S \circ \omega_\K \, .
\]
\end{thm}

In fact we prove the ``if'' part of Theorem \ref{thm:main} without the normalcy condition on $\K$.

The proof of Theorem \ref{thm:main} depends on a structural result for characteristic idempotents of saturated fusion systems with normal subgroups that is interesting in its own right (Lemma \ref{lem:normal case}). The proof of Conjecture \ref{running-conjecture} will likely depend on more general structural results for characteristic idempotents and we postpone tackling the general case until those tools are available.

Specializing to the case where $\K$ is $O^p(\F)$ or $O^{p'}(\F)$ (see \cite{BCGLO2} for definitions), we get the following corollary. Note that part (1) was used in \cite{DGPS:YoshidaAndTate} as Theorem 4.1.

\begin{cor}
Let $\F$ be a saturated fusion system on a finite $p$-group $S$.  Then
\begin{enumerate}
\item $\omega_\F \circ [T,\incl]^S_T = [T,\incl]^S_T \circ \omega_{O^p(\F)}$, where $T$ is the hyperfocal subgroup of $\F$.
\item $\omega_\F = \omega_{N_\F(S)} \circ \omega_{O^{p'}(\F)}$.
\end{enumerate}
\end{cor}

Using the conventions for Mackey functors on fusion systems laid out in Section \ref{sec:Mackey}, we have the following consequence of Theorem \ref{thm:main}.

\begin{cor} \label{cor:Mackey}
Let $\F, \H$ and $\K$ be as in Theorem \ref{thm:main}. If $M$ is a $p$-local, $p$-defined Mackey functor, then the maps 
\[ M(\H) \xrightarrow{tr_{\H}^{\F}} M(\F) \xrightarrow{res_{\K}^{\F}} M(\K) \]
and
\[ M(\H) \xrightarrow{res_T^\H} M(T) \xrightarrow{tr_T^\K} M(\K) \]
agree.
\end{cor}
The maps in Corollary \ref{cor:Mackey} correspond to the maps in \eqref{map:HGK} and \eqref{map:HLK}. The analogue to the latter map necessarily takes a different form in the fusion setting as we let it factor through $M(T)$ rather than the value of $M$ at the intersection of $\H$ and $\K$, since this intersection is not necessarily a saturated fusion system. However, the map actually factors through the submodule of elements in $M(T)$ that are both $\H$-stable and $\K$-stable, which is analogous to factoring through the intersection.

This project was initiated when the first two authors met at the Ohio State University in the summer of 2009. We are grateful to OSU for their hospitality and to the Midwest Topology Network for the travel funding that made this visit possible.

\section{Notation and terminology} \label{sec:Notation}

\subsection{Bisets and the double Burnside ring} \label{subsec:Burnside}
For finite groups $G$ and $H$ a \emph{$(G,H)$-biset} is a set with left $H$-action and right $G$-action such that the two actions commute. We say that a $(G,H)$-biset is \emph{left-free} if the left action is free, \emph{right-free} if the right action is free, and \emph{bifree} if both actions are free. Isomorphism classes of finite, left-free $(G,H)$-bisets form a monoid with cancellation under disjoint union, and the \emph{double Burnside module} $A(G,H)$ is the group completion of this monoid. 

The $\Z$-module structure of $A(G,H)$ is easy to describe. For a subgroup $K \leq G$ and a homomorphism $\varphi \colon K \to H$ one obtains an indecomposable $(G,H)$-biset 
\[  
  H \times_{K,\varphi} G = H \times G / \sim \, ,
\]
where $\sim$ is the equivalence relation $(h,kg)\sim (h\varphi(k),g)$ for $k \in K$. The isomorphism class of $H \times_{K,\varphi} G$ is determined by the conjugacy class of $K$ and $\varphi$, with conjugacy taken in both $G$ and $H$, and we denote it by $[K,\varphi]^H_G$, or just $[K,\varphi]$ when there is no danger of confusion. It is not hard to check that every indecomposable $(G,H)$-biset belongs to an isomorphism class $[K,\varphi]$, and it follows that $A(G,H)$ is a free $\Z$-module with basis the collection of elements $[K,\varphi]$ running over all conjugacy classes of pairs $(K,\varphi)$. We refer to this basis as the \emph{standard basis} of $A(G,H)$. For an element $X \in A(G,H)$ we let $\chi_{_{[K,\varphi]}}(X)$ denote the coefficient of $[K,\varphi]$ in the standard basis decomposition, so we have
\[ X = \sum_{[K,\varphi]} \chi_{_{[K,\varphi]}}(X) \, [K,\varphi] \, . \]
For a $\Z$-module $M$ we write $M_{(p)}$ for the $p$-localization of $M$. The standard basis of $A(G,H)$ also forms a basis for $A(G,H)\pLoc$ and we denote the $p$-localization of the morphisms $\chi_{_{[K,\varphi]}}$ also by $\chi_{_{[K,\varphi]}}$. 

Another convenient way to characterize elements in $A(G,H)$ is by their fixed points. For a $(G,H)$-biset $X$ and a pair $(K,\varphi)$ with $K \leq G$ and $\varphi \colon  K \to H$, set
\[ \Phi_{\langle K,\varphi \rangle } (X) = | \{ x \in X \mid \forall k \in K \colon xk = \varphi(k)x \} | \, .  \]
Extending linearly, we get a homomorphism 
\[ \Phi_{\langle K,\varphi \rangle } \colon A(G,H) \to \Z \, . \]
A classical result, going back to Burnside, says that the product of maps $\Phi_{\langle K,\varphi \rangle }$, running over all conjugacy classes of pairs $\langle K,\varphi \rangle$, is an injection. In other words, $X = Y$ in $A(G,K)$ if and only if $\Phi_{\langle K,\varphi \rangle } (X) = \Phi_{\langle K,\varphi \rangle } (Y) $ for all $\langle K,\varphi \rangle $.

Burnside modules admit a composition pairing
 \[ A(H,K) \times A(G,H) \longrightarrow A(G,K) \, , \quad (X,Y) \longmapsto X\circ Y \, , \]
defined by setting $(X \circ Y) = X \times_H Y$ for bisets $X$ and $Y$ and then extending linearly to general elements. By definition the composition pairing is bilinear and in particular $A(G,G)$ is a ring, called the \emph{double Burnside ring of $G$}. This ring has unit $[G,\id]$, which is the isomorphism class of $G$ regarded as a $(G,G)$-biset by translation. We endow the double Burnside ring with an \emph{augmentation} $\epsilon \colon A(S,S) \to \Z$, defined on basis elements by $\epsilon([P,\varphi]_S^S) = |S/P|$.

\subsection{Characteristic idempotents}
Characteristic idempotents play the role for fusion systems that the $(S,S)$-biset $G$ plays for a finite group $G$ with Sylow $p$-subgroup $S$. As described in Section \ref{sec:Mackey}, this provides a way to define Mackey functors and transfers on fusion systems. Saturated fusion systems on a $p$-group $S$ can in fact be represented in the $p$-localized double Burnside ring $A(S,S)\pLoc$ by their characteristic idempotents. A detailed account of this correspondence is given in \cite{KRStancu:SFSasIdem}, and here we only recall the definitions as needed. 

\begin{defn}
Let $\F$ be a fusion system on a finite $p$-group $S$, and let $X \in A(S,S)\pLoc$. We say that $X$ is \emph{$\F$-generated} if $X$ is a linear combination of standard basis elements $[P,\varphi]$ with $\varphi$ in $\F$.
\end{defn}

\begin{defn}
Let $\F$ be a fusion system on a finite $p$-group $S$, and let $H$ be any finite group. 
\begin{itemize}
  \item We say that $X \in A(S,H)\pLoc$ is \emph{right $\F$-stable} if, for every $P \leq S$ and every $\varphi \in \Hom_\F(P,S)$, the following equation holds in $A(P,H)\pLoc$:
\[ X \circ [P,\varphi]_P^S = X \circ [P,\incl]_P^S \, . \]
 \item We say that $X \in A(H,S)\pLoc$ is \emph{left $\F$-stable} if, for every $P \leq S$ and every $\varphi \in \Hom_\F(P,S)$, the following equation holds in $A(H,P)\pLoc$:
\[ [\varphi(P),\varphi^{-1}]_S^P \circ X  = [P,\id]_S^P \circ X \, .\]
\end{itemize}
\end{defn}

\begin{defn}
Let $\F$ be a saturated fusion system on a finite $p$-group $S$. An element $\Omega \in A(S,S)\pLoc$ is a \emph{characteristic element} for $\F$ if it is $\F$-generated, left and right $\F$-stable, and has augmentation prime to $p$. 
A characteristic element for $\F$ that is idempotent is called a \emph{characteristic idempotent} for $\F$.
\end{defn}

It is easy to check that for a finite group $G$ with Sylow $p$-subgroup $S$, the $(S,S)$-biset $G$ is a characteristic element for $\F_S(G)$. It was shown in \cite{BLO2} that every saturated fusion system has a characteristic element, and in \cite{KR:ClSpectra} that every saturated fusion system has a unique characteristic idempotent. We denote the characteristic idempotent of a saturated fusion system $\F$ by $\omega_\F$.

The following result from \cite{KR:ClSpectra} is crucial to our arguments.
\begin{prop}
Let $\F$ be a fusion system on a finite $p$-group $S$, and let $H$ be any finite group. 
\begin{itemize}
  \item[(a)] $X \in A(S,H)\pLoc$ is right $\F$-stable if and only if $X \circ \omega_\F = X$.
 \item[(b)] $X \in A(H,S)\pLoc$ is left $\F$-stable if and only if $\omega_\F \circ X = X$.
\end{itemize}
\end{prop}

We will also need the following characterization of $\F$-stability in terms of fixed points from \cite{KRStancu:SFSasIdem}.
\begin{lem} \label{lem:RephraseStable}
Let $\F$ be a fusion system on a finite $p$-group $S$, and let $H$ be any finite group. 
\begin{itemize}
 \item[(a)] $X \in A(S,H)\pLoc$ is right $\F$-stable if and only if for every pair $(Q,\psi)$ with $Q \leq S$ and $\psi \colon Q \to H$, and every $\varphi \in \Hom_{\F}(Q,S)$, 
  \[ \Phi_{\langle Q,\psi \rangle} (X) = \Phi_{\langle \varphi(Q),\psi \circ \varphi^{-1} \rangle} (X)\,. \]
 \item[(b)] $X \in A(H,S)\pLoc$ is left $\F$-stable if and only if for every pair $(Q,\psi)$ with $Q \leq H$ and $\psi \colon Q \to S$, $\varphi \in \Hom_{\F}(\psi(Q),S)$,
  \[ \Phi_{\langle Q,\psi \rangle} (X) = \Phi_{\langle Q, \varphi \circ \psi \rangle} (X)\,. \]
\end{itemize}
\end{lem}
Lemma \ref{lem:RephraseStable} is proven as Lemma 4.8 of \cite{KRStancu:SFSasIdem} in the case where $H = S$, but the argument holds for general $H$.

\section{Characteristic idempotents of composition products} \label{sec:Conj}
We now restrict attention to the case where $R=S$. Under this assumption, Conjecture~\ref{running-conjecture} becomes

\begin{conj} \label{conj-right-side}
Let $\F$ be a saturated fusion system on a finite $p$-group $S$.  Let $\H$ be a saturated subsystem of $\F$ on $S$, and let $\K$ be a saturated subsystem of $\F$ on some subgroup $T$ of $S$.  Then we have ${}_{S}\F_T = \H\K$ if and only if
\begin{equation} \label{E:idemp-relation}
	\omega_\F \circ [T,\incl]_T^S = \omega_\H \circ [T,\incl]_T^S \circ \omega_\K \, . \tag{$\ast$}
\end{equation}
\end{conj}

This special case of Conjecture~\ref{running-conjecture} admits a convenient description.

\begin{prop} \label{prop-right-side}
Let $\F$ be a saturated fusion system on a finite $p$-group $S$.  Let $\H$ be a saturated subsystem of $\F$ on $S$, and let $\K$ be a saturated subsystem of $\F$ on some subgroup $T$ of $S$.  Suppose ${}_{S}\F_T = \H\K$.  Then the condition~\eqref{E:idemp-relation} in Conjecture~\ref{conj-right-side} is equivalent to either one of the following.
\begin{enumerate}
\item  $[T,\id]_S^T \circ \omega_\H \circ [T,\incl]_T^S \circ \omega_\K$ is left $\K$-stable.
\item $\omega_\K \circ [T,\id]_S^T \circ \omega_\H \circ [T,\incl]_T^S$ is right $\K$-stable.
\end{enumerate}
\end{prop}

\begin{proof}
We have
\[
	\omega_\F \circ [T,\incl]_T^S = \omega_\F \circ [T,\incl]_T^S\circ \omega_\K = \omega_\F \circ \omega_\H \circ [T,\incl]_T^S\circ \omega_\K,
\]
where the first equality follows from that $\omega_\F$ is right $\K$-stable, and the second equality follows from that $\omega_\F$ is right $\H$-stable.  Thus \eqref{E:idemp-relation} is equivalent to that $X:=\omega_\H \circ [T,\incl]_T^S\circ \omega_\K$ is left $\F$-stable, or equivalently
\[
	\Phi_{\la P,\varphi \ra}(X) = \Phi_{\la P,\incl \ra}(X)
\]
for every $P\leq T$ and $\varphi\in\Hom_\F(P,S)$.  So let $P\leq T$ and $\varphi\in\Hom_\F(P,S)$.  By assumption, we have $\varphi=\eta\circ\kappa$ for some $\eta$ in $\H$ and $\kappa$ in $\K$. Then
\begin{align*}
	\Phi_{\la P,\varphi \ra}(X) 
	&= \Phi_{\la P,\eta\circ\kappa \ra}(X) \\
	&= \Phi_{\la \kappa(P),\eta \ra}(X) \qquad\text{($\because$ $X$ is right $\K$-stable)} \\
	&= \Phi_{\la \kappa(P), \incl \ra}(X) \qquad\text{($\because$ $X$ is left $\H$-stable)} \\
	&= \Phi_{\la P, \incl\circ\kappa \ra}(X) \qquad\text{($\because$ $X$ is right $\K$-stable)} ,
\end{align*}
where the last inclusion map is $T \hookrightarrow S$. This shows that \eqref{E:idemp-relation} is equivalent to that
\[
	\Phi_{\la P,\kappa \ra}([T,\id]_S^T\circ X) = \Phi_{\la P,\incl\circ\kappa \ra}(X) = \Phi_{\la P,\incl \ra}(X)
\]
for every $P\leq T$ and $\kappa\in\Hom_\K(P,T)$, or equivalently
\[
	[T,\id]_S^T \circ X = [T,\id]_S^T \circ \omega_\H \circ [T,\incl]_T^S \circ \omega_\K
\]
is left $\K$-stable.  The other equivalent condition is obtained by taking opposite to the above.
\end{proof}

\section{The normal case} \label{sec:Proofs}

We confirm Conjecture~\ref{conj-right-side} positively with an additional normalcy condition.  Let $\F$ be a saturated fusion system on a finite $p$-group $S$.  Let $\H$ be a saturated subsystem of $\F$ on $S$, and let $\K$ be a saturated subsystem of $\F$ on some subgroup $T$ of $S$. Recall, see \cite{AschNormal,Craven}, that $\K$ is \emph{weakly normal} in $\F$ if $T$ is a strongly $\F$-closed subgroup of $S$ and, whenever $\varphi\in\Hom_\F(P,S)$ and $\psi\in\Hom_\K(Q,R)$, for $Q,R\le P$, we have $\varphi\psi\varphi^{-1}\in\Hom_\K(\varphi(Q),\varphi(R))$.

\begin{lem} \label{lem:normal case}
Suppose that $\K$ is weakly normal in $\F$. Then the following are equivalent.
\begin{enumerate}
\item ${}_{S}\F_T = \H\K$
\item $\Aut_\F(T)=\Aut_\H(T)\Aut_\K(T)$. 
\item ${}_{S}\F_T = N_\H(T)\K$
\end{enumerate}
\end{lem}

\begin{proof}
This follows from the Frattini property of weakly normal subsystems (see \cite[\S 3]{AschNormal}).
\end{proof}

\begin{lem}\label{lem:idempotent-for-normalizer}
The following are equivalent.
\begin{enumerate}
\item $\F=N_\F(U)$ for some $U\leq S$.
\item $\chi_{_{[P,\varphi]}}(\omega_\F) = 0$ unless $P\geq U$ and $\varphi$ belongs to $N_\F(U)$. 
\end{enumerate}
\end{lem}

\begin{proof}
(1) $\Rightarrow$ (2): Let us use the notation in Proposition 5.6 of \cite{KR:ClSpectra}.  In particular
\[
	\omega=\omega_\F = \sum_{0\leq i\leq n}\sum_{0\leq j\leq m_i} c_{ij}[P_i,\varphi_{ij}],\qquad c_{ij} \in \Z_{(p)}.
\]
Write 
\[
	\omega = \omega_{\geq U} + \omega_{\not\geq U}
\]
where
\begin{gather*}
	\omega_{\geq U} = \sum_{i,j \text{ s.t.\ } P_i\geq U} c_{ij}[P_i,\varphi_{ij}],\\
	\omega_{\not\geq U} = \sum_{i,j \text{ s.t.\ } P_i\not\geq U} c_{ij}[P_i,\varphi_{ij}].
\end{gather*}
We claim that both $\omega_{\geq U}$ and $\omega_{\not\geq U}$ are right $\F$-stable.  Since $\F=N_\F(U)$, an element $X\in A(S,S)_{(p)}$ is right $\F$-stable if and only if
\[
	X \circ [Q,\psi]^S_Q = X \circ [Q,\incl]^S_Q
\]
for all $[Q,\psi]^S_Q$ with $Q\geq U$ and $\psi$ in $\F$. Suppose $U\leq Q\leq S$, $\psi\in\Hom_\F(Q,S)$. Then
\[
	[P_i,\varphi_{ij}]^S_S \circ [Q,\psi]^S_Q = \sum_{x \in [P_i \backslash S / \psi(Q)]} [\psi^{-1}(\psi(Q) \cap P_i^{x}), \varphi_{ij} \circ c_x \circ \psi ]^S_Q,
\]
and $\psi^{-1}(\psi(Q) \cap P_i^{x}) \geq U$ if and only if $P_i \geq U$.  This shows that the right $\F$-stability of $\omega$ implies the right $\F$-stability of $\omega_{\geq U}$ and $\omega_{\not\geq U}$.

Now the right $\F$-stability of $\omega_{\geq U}$ implies that $\omega_{\geq U}$ satisfies the equations
\[
	\sum_{k,l} c_{kl} \left(|[P_k,\varphi_{kl}]^{\Delta_i^j}|-|[P_k,\varphi_{kl}]^{\Delta_i^j}|\right) = 0 \tag{$ij$}.
\]
Also, by Lemma 5.5 of \cite{KR:ClSpectra}, $\omega_{\geq U}$ satisfies the equations
\begin{gather*}
	\sum_{j=0}^{m_0} c_{0j} = 1 \tag{$00$},\\
	\sum_{j=0}^{m_i} c_{ij} = 0 \tag{$i0$}.
\end{gather*}
Thus, by Proposition 5.6 of \cite{KR:ClSpectra}, we have $\omega=\omega_{\geq U}$.

(2) $\Rightarrow$ (1): Suppose (2).  Then 
\[
	\PreOrb(\omega_\F) \subseteq N_\F(U),
\]
and so
\[
	\Orb(\omega_\F) \subseteq N_\F(U).
\]
Since $\Orb(\omega_\F)=\F$ by Proposition 5.10 of \cite{KRStancu:SFSasIdem}, we have $\F=N_\F(U)$.
\end{proof}


We can now prove our main result.

\begin{proof}[Proof of Theorem \ref{thm:main}]
First suppose that ${}_{S}\F_T = \H\K$.  By Proposition~\ref{prop-right-side}, we need to show that $X:=\omega_\K \circ [T,\id]_S^T \circ \omega_\H \circ [T,\incl]_T^S$ is right $\K$-stable.  By Lemma~\ref{lem:normal case}, we may assume that $\H=N_\H(T)$.  Hence, by Lemma~\ref{lem:idempotent-for-normalizer}, we have 
\[
	\omega_\H = \sum_{[Q,\psi]}c_{[Q,\psi]}[Q,\psi]
\]
where $c_{[Q,\psi]}\in\Z_{(p)}$ and the sum is taken over the $[Q,\psi]$ with $Q \geq T$ and $\psi$ in $\H$. Using that $T$ is strongly $\F$-closed, we get
\begin{align*}
	X
	&=\omega_\K \circ [T,\id]_S^T \circ \sum_{[Q,\psi]}c_{[Q,\psi]}[Q,\psi]^S_S \circ [T,\incl]_T^S\\
	&=\sum_{[Q,\psi]} c_{[Q,\psi]} \omega_\K \circ [T,\id]_S^T \circ [Q,\psi]^S_S \circ [T,\incl]_T^S\\
	&=\sum_{[Q,\psi]} \sum_{x \in [S/\psi(Q)]} c_{[Q,\psi]} \omega_\K \circ [T,c_x\circ\psi]^T_S \circ [T,\incl]_T^S\\	
	&=\sum_{[Q,\psi]} \sum_{x \in [S/\psi(Q)]} \sum_{y\in[S/T]} c_{[Q,\psi]} \omega_\K \circ [T,c_x\circ\psi\circ c_y]^T_T \, .
\end{align*}
Since $\K$ is weakly normal in $\F$, using Theorem 8.2 of \cite{KRStancu:SFSasIdem} we have
\begin{align*}
	X\circ\omega_\K
	&=\sum_{[Q,\psi]} \sum_{x \in [S/\psi(Q)]} \sum_{y\in[S/T]} c_{[Q,\psi]} \omega_\K \circ [T,c_x\circ\psi\circ c_y]^T_T\circ\omega_\K \\
	&=\sum_{[Q,\psi]} \sum_{x \in [S/\psi(Q)]} \sum_{y\in[S/T]} c_{[Q,\psi]} \omega_\K \circ\omega_\K \circ [T,c_x\circ\psi\circ c_y]^T_T \\
	&=\sum_{[Q,\psi]} \sum_{x \in [S/\psi(Q)]} \sum_{y\in[S/T]} c_{[Q,\psi]} \omega_\K \circ [T,c_x\circ\psi\circ c_y]^T_T = X,\\
\end{align*}
as desired.

Conversely, suppose that $\omega_\F \circ [T,\incl]_T^S = \omega_\H \circ [T,\incl]_T^S \circ \omega_\K$.  Let $\varphi\in\Hom_\F(P,S)$ for some $P\leq T$.  Then $\Phi_{\la P,\varphi \ra}(\omega_\F) \neq 0$.  Clearly $\Phi_{\la P,\varphi \ra}(\omega_\F) = \Phi_{\la P,\varphi \ra}(\omega_\F \circ [T,\incl]^S_T)$.  So by assumption we have $\Phi_{\la P,\varphi \ra}(\omega_\H \circ [T,\incl]_T^S \circ \omega_\K) \neq 0$.  But $\omega_\H \circ [T,\incl]_T^S \circ \omega_\K$ is a linear combination of terms of the form $[Q,\eta]^S_S \circ [T,\incl]^S_T \circ [R,\kappa]^T_T$ where $Q\leq S$, $R\leq T$, $\eta$ in $\H$ and $\kappa$ in $\K$.  Applying  the double coset formula to those terms, we see that $\varphi = \eta|_{x\kappa(P)x^{-1}} \circ c_x \circ \kappa|_P$ for some $\eta$ in $\H$, $\kappa$ in $\K$ and $x\in S$ and we are done.
\end{proof}

\section{Mackey functors for fusion systems} \label{sec:Mackey}
The composition pairing described in \ref{subsec:Burnside} allows us to form the $\Z$-linear category $\Burnside$, called the \emph{Burnside category}, with objects the finite groups and morphism sets $\mor_\Burnside(G,H) = A(G,H)$. We take a \emph{Mackey functor} to mean a functor $M \colon \Burnside \to \Z\module$. This functor can be contravariant or covariant; we will focus on the contravariant case, leaving the necessary adjustments for the covariant case to the reader. A popular example of a contravariant Mackey functor are the cohomology functors $H^k(-,A)$, $k \geq 0$ with coefficients in any abelian group $A$.

The notion of Mackey functor used here is sometimes called a globally-defined Mackey functor. (A ``classical'' Mackey functor is defined on the subgroups of a given group). We also need to consider more restrictive functors, defined only on certain groups. Let $\Burnside_p$ be the full subcategory of $\Burnside$ with objects the finite $p$-groups. A \emph{$p$-defined Mackey functor} is a functor $\Burnside_p \to \Z\module$.

The \emph{$p$-local Burnside category} is the $\Z\pLoc$-linear category $\Burnside\Z\pLoc$ obtained by $p$-localizing the morphism modules in $\Burnside$. A \emph{$p$-local Mackey functor} is a functor $\Burnside\Z\pLoc \to \Z\pLoc\module$, which we observe is the same as a functor $\Burnside \to \Z\pLoc\module$. Similarly, a \emph{$p$-local, $p$-defined Mackey functor} is a functor $\Burnside_p\Z\pLoc \to \Z\pLoc\module$

A Mackey functor $M$ is \emph{$p$-projective} if for every finite group $G$, the canonical map 
\[ \bigoplus_{P \leq G} M(P) \to M(G) \, ,\]
where the sum runs over all $p$-subgroups, is a surjection. By Dress \cite{Dress}, $p$-projectivity implies that $M(G)$ is determined $p$-locally, which in the language of fusion systems means it is determined by its values on the fusion system $\F_S(G)$ when $S$ is a Sylow subgroup of $G$. More precisely, the composite 
\[ M(G) \xrightarrow{res_S^G} M(S) \xrightarrow{tr_S^G} M(G) \]
is an isomorphism that factors isomorphically through the submodule of elements in $M(S)$ that are $\F_S(G)$-stable in the sense of definition \ref{def:MStable} below. The opposite composite 
\[ M(S) \xrightarrow{tr_S^G} M(G) \xrightarrow{res_S^G} M(S) \]
is the map induced by the $(S,S)$-biset $G$ and has image isomorphic to $M(G)$. Thus, when $M$ is $p$-projective, $M(G)$ is determined by $M(S)$ and the $(S,S)$-biset $G$.

\begin{defn} \label{def:MStable}
Let $\F$ be a fusion system on a finite $p$-group $S$. For a (contravariant) $p$-defined Mackey functor $M$, we say that an element $x \in M(S)$ is \emph{$\F$-stable} if, for every $P \leq S$ and every $\varphi \in \Hom_\F(P,S)$, we have $M(\varphi)(x) = M(\incl)(x)$ in $M(P)$. 
\end{defn}

The preceding discussion suggests the following approach for extending Mackey functors to fusion systems. Given a $p$-local, $p$-defined Mackey functor $M$ and a fusion system $\F$ on a finite $p$-group $S$, define $M(\F)$ as the module of $\F$-stable elements in $M(S)$. We can then regard the inclusion of $M(\F)$ in $M(S)$ as a restriction map, which we denote by $res_S^\F$, and the following theorem allows us to restrict the map $M(\omega_\F) \colon M(S) \to M(S)$ induced by the characteristic idempotent to a map $ M(S) \xrightarrow{tr^\F_S} M(\F)$ that plays the role of transfer.

\begin{thm}[\cite{KRStancu:SFSasIdem}]
Let $M$ be a (contravariant) $p$-local, $p$-defined Mackey functor, and let $\F$ be a saturated fusion system on a finite $p$-group $S$. An element $x \in M(S)$ is $\F$-stable if and only if $M(\omega_\F)(x) = x$.
\end{thm}

Observe that, by construction, the composite 
\[ M(\F) \xrightarrow{res_S^\F} M(S) \xrightarrow{tr_S^\F} M(\F) \] 
is the identity, and the opposite composite
\[ M(S) \xrightarrow{tr_S^\F} M(\F) \xrightarrow{res_S^\F} M(S)\]
is the idempotent map $M(\omega_\F)$. It follows that 
\[ M(\omega_\F) \circ res_S^\F = res_S^\F \quad \text{and} \quad tr_S^\F \circ M(\omega_\F) = tr_S^\F \, ,\]
which we will use in the proof of Corollary \ref{cor:Mackey} given below.

\begin{rmk}
If one starts with $p$-local Mackey functor $M$, restricts to a $p$-local, $p$-defined Mackey functor and then extends to fusion systems, one should not expect $M(G)$ to be isomorphic to $M(\F_S(G))$ for a finite group $G$ with Sylow $p$-subgroup $S$ unless $M$ is $p$-projective. In some sense this process results in a $p$-projective version of $M$, which retains only information detectable on $p$-groups.
\end{rmk}

For a fusion system $\F$ on a finite $p$-group $S$ and a subsystem $\K$ on a subgroup $T \leq S$, one obtains transfer and restriction maps via the composites
\[ tr_\K^\F \colon M(\K) \xrightarrow{res_T^\K} M(T) \xrightarrow{tr^S_T} M(S) \xrightarrow{tr_S^\F} M(\F) \, ,\]
and
\[ res_\K^\F \colon M(\F) \xrightarrow{res_S^\F} M(S) \xrightarrow{res_T^S} M(T) \xrightarrow{tr_T^\K} M(\K) \, . \] 
The arguments in Theorems 7.9 and 8.6 of \cite{KR:ClSpectra} show that these behave well under composition (so $tr_\K^\F \circ tr_\H^\K = tr_\H^\F$ and $res_\H^\K \circ res_\K^\F = res_\H^\F$ when $\H$ is a subsystem of $\K$).

\begin{proof}[Proof of Corollary \ref{cor:Mackey}]
From Theorem \ref{thm:main} we have
\[ tr_T^\K \circ M([T,\incl]_T^S) \circ M(\omega_\F) \circ res_S^\H = tr_T^\K \circ M(\omega_\K) \circ M([T,\incl]_T^S) \circ M(\omega_\H) \circ res_S^\H \, .\]
Rewriting on the left, we get
\begin{align*} 
   tr_T^\K \circ M([T,\incl]) \circ M(\omega_\F) \circ res_S^\H 
&= tr_T^\K \circ res_T^S \circ (res_S^\F \circ tr_S^\F) \circ res_S^\H \\
&= (tr_T^\K \circ res_T^S \circ res_S^\F) \circ (tr_S^\F \circ res_S^\H) \\
&= res_\K^\F \circ tr_\H^\F \, ,
\end{align*}
while rewriting on the right yields
\begin{align*}
  tr_T^\K \circ M(\omega_\K) \circ M([T,\incl]_T^S) \circ M(\omega_\H) \circ res_S^\H
&= (tr_T^\K \circ M(\omega_\K)) \circ res_T^S \circ (M(\omega_\H) \circ res_S^\H) \\
&= tr_T^\K \circ res_T^S \circ res_S^\H\\
&= tr_T^\K \circ res_T^\H \, ,
\end{align*}
and so the two maps agree.
\end{proof}


\end{document}